\documentclass[a4paper]{article}
\usepackage{RR}

\usepackage{graphicx}
\usepackage{amsmath}
\usepackage{amssymb}
\usepackage[latin1]{inputenc}
\usepackage{here}
\usepackage{verbatim}

\usepackage[a4paper, left=2.3cm,top=2.3cm,right=2.3cm,bottom=2.3cm]{geometry}

\newcommand{\precu}{{\bar U}}
\newcommand{\precl}{{\bar L}}
\newcommand{\precd}{{\bar D}}
\newcommand{\precf}{{\bar F}}
\newcommand{\precc}{{\bar C}}

\newcommand{\be}{{\bf e}}

\def \be{\begin{equation}}
\def \ee{\end{equation}}
\def \bea{\begin{eqnarray}}
\def \eea{\end{eqnarray}}
\def \bean{\begin{eqnarray*}}
\def \eean{\end{eqnarray*}}

\newtheorem{lemma}{Lemma}[section]

\newtheorem{definition}{Definition}[section]
\newenvironment{proof}[1][Proof]{\textbf{#1.} }{\ \rule{0.5em}{0.5em}}

\RRtitle{D\'ecomposition \`a base de filtrage g\'en\'eralis\'ee}
\RRetitle{Generalized Filtering Decomposition} 

\RRauthor{Laura Grigori\thanks{INRIA Saclay - Ile de France, Laboratoire
    de Recherche en Informatique Universite Paris-Sud 11, France ({\tt
      laura.grigori@inria.fr}). } \and Frederic Nataf
  \thanks{Laboratoire J.L. Lions, CNRS UMR7598, Universite Paris 6,
    France ({\tt nataf@ann.jussieu.fr}).}  }

\RRabstract{This paper introduces a new preconditioning technique that
  is suitable for matrices arising from the discretization of a system
  of PDEs on unstructured grids.  The preconditioner satisfies a
  so-called filtering property, which ensures that the input matrix is
  identical with the preconditioner on a given filtering vector. This
  vector is chosen to alleviate the effect of low frequency modes on
  convergence and so decrease or eliminate the plateau which is often
  observed in the convergence of iterative methods.  In particular,
  the paper presents a general approach that allows to ensure that the
  filtering condition is satisfied in a matrix decomposition.  The
  input matrix can have an arbitrary sparse structure.  Hence, it can
  be reordered using nested dissection, to allow a parallel
  computation of the preconditioner and of the iterative process.}

\RRresume{Ce document pr\'esente une nouvelle technique de
  pr\'econditionnement adapt\'e pour les matrices issues de la
  discr\'etisation d'un système d'\'equations aux d\'eriv\'ees
  partielles sur des maillages non structur\'es.  Le préconditionneur
  satisfait une propri\'et\'e dite de filtrage, qui signifie que la
  matrice d'entr\'ee est identique au pr\'econditionneur pour un
  vecteur donn\'e de filtrage.  Le choix de ce vecteur permet
  d'att\'enuer l'effet des modes de basse fr\'equence sur la
  convergence et ainsi de diminuer ou d'\'eliminer le plateau qui est
  souvent observ\'e dans la convergence des m\'ethodes it\'eratives.
  En particulier, le document pr\'esente une approche g\'en\'erale qui
  permet d'assurer que la propri\'et\'e de filtrage est satisfaite
  lors d'une d\'ecomposition matricielle.  La matrice d'entr\'ee peut
  avoir une structure creuse arbitraire. Ainsi, elle peut \^etre
  r\'enum\'erot\'ee en utilisant la m\'ethode de dissection
  embo\^it\'ee, afin de permettre un calcul parall\`ele du
  pr\'econditionneur et du processus it\'eratif.}

\RRkeyword{linear solvers, Krylov subspace methods, preconditioning, filtering property, block incomplete decomposition}

\RRprojet{Grand-large}
\RRmotcle{solveurs lin\'eaires, m\'ethode de sous-espaces de Krylov, pr\'econditionnement, propri\'et\'e de filtrage, d\'ecomposition incompl\`ete par blocs}
\RRtheme{\THNum}
\URFuturs
\RCSaclay

\begin{document}
\RRNo{7569}
\makeRR

\section{Introduction} 

Iterative methods are widely used in industrial applications, and
preconditioning these methods is an important topic which has already
been extensively studied \cite{MR2000f:65003, Saad:1996:IMS,
  Benzi:1999:ACS}.  In this context, algebraic multigrid methods are
very successful for certain classes of applications, in particular
scalar PDEs \cite{MR972756, MR1820878, MR1807961,
  MR1370108, MR1835471}.  They are known to have good weak scalability
properties, but they are not strongly scalable.  This motivates
research on iterative solvers for systems of PDEs and/or large number
of processors.

Several highly used preconditioners, as the incomplete LU
factorizations and domain decomposition methods, are known to have
scalability problems, in terms of both problem size and number of
processors. This is often due to the presence of several low frequency
modes that hinder the convergence of the iterative method.  To solve
this problem, a different class of so called filtering preconditioners
has been proposed \cite{MR2008372, Achdou:2007:LFT, Wagner:1997:TFFU,
  Wagner:1997:AF}, where the choice of the filtering vector is made to
alleviate the effect of low frequency modes on the convergence.  For
domain decomposition methods, coarse grid correction is known to be
mandatory for solving the scalability problem \cite{MR2104179}.

In this paper we focus on the generalization and suitability for
parallel computing of the filtering preconditioner.  This
preconditoner is an incomplete factorization where it is possible to
ensure that the factorization will coincide with the original matrix
for some specified vector, called a filtering vector. Satisfying this
filtering condition is an important factor for accelerating the
convergence of the iterative method.  The previous research on these
methods considered only matrices arising from the discretization of
PDEs on structured grids, where the matrix has a block tridiagonal
structure \cite{MR2008372, Achdou:2007:LFT, Wagner:1997:TFFU,
  Wagner:1997:AF}. To the best of our knowledge, there was no previous
result on the parallelization of filtering preconditioners.  One of
the important results of this research is the development of a new and
general approach to ensure that a filtering condition is satisfied in
a matrix decomposition. This approach is based on an innovative way of
organizing the computations that allows on one side to satisfy a
filtering property and on another side to perform a parallel
computation. This approach has been used to develop a preconditioner
based on a block approached decomposition, that we refer to as block
filtering preconditioner.  While we discuss in detail the right
filtering property $A t = Mt$, a similar approach can be used to
develop a preconditioner that satisfies the left filtering property
$t^T A = t^T M$, where $A$ is the input matrix, $M$ is the
preconditioner and $t$ is the filtering vector.

This preconditioner does not impose any particular structure on the
input matrix.  To allow its usage on parallel architectures, the input
matrix can be reordered using nested dissection.  This reordering
allows a parallel implementation of the construction of the
preconditioner, as well as of the iterative process.

The preconditioner can be seen as a generalization for unstructured
grids of the preconditioner presented in \cite{Achdou:2007:LFT} for
block tridiagonal matrices.  In contrast to the preconditioner
presented in \cite{Achdou:2007:LFT} that has been shown to be
efficient in combination with ILU0, the block preconditioner presented
here is efficient as a stand-alone preconditioner.

The goal of this
paper is only to present the algebraic framework which allows a
filtering condition to be satisfied in an incomplete block
factorization. The numerical results showing the efficiency of the
proposed preconditioner and its parallel performance will be presented
in a future paper.

\section{Block Filtering Decomposition}
\label{sec:blockfilter}

In this section we describe a block filtering preconditioner $M$ which
satisfies the right filtering condition $(M-A)t =0$, where $t$ is a
filtering vector.

Consider a matrix $A$ of size $n \times n$ partitioned into a block
matrix of size $N \times N$ with square diagonal blocks (not
necessarily of a same size)
\[
A = 
  \begin{pmatrix}
    A_{11} & \ldots  &  A_{1N} \\
    \vdots & \ddots & \vdots \\
    A_{N1} & \ldots  &  A_{NN} \\
  \end{pmatrix}.
\]
An exact block $LDU$ factorization of $A$ is written as {\small
\begin{equation}
     \label{eq:LDU}
A =\left(\begin{array}{cccc}
       D_{11}    &  &          &          \\
       L_{21} & D_{22}    &   &          \\
         \vdots     & \ddots & \ddots   &  \\
       L_{N1}       & \ldots       & L_{N,N-1} & D_{NN}
       \end{array}
     \right)
\left(\begin{array}{cccc}
       D_{11}^{-1}    &  &          &          \\
        & D_{22}^{-1}    &   &          \\
              & & \ddots   &  \\
              &        &  & D_{NN}^{-1}
       \end{array}
     \right)
\left(\begin{array}{cccc}
       D_{11}     & U_{12} &  \ldots  & U_{1N} \\
        & \ddots  & \ddots  & \vdots \\
        &  & D_{N-1,N-1}    & U_{N-1,N} \\
              &        &  &  D_{NN}
       \end{array}
     \right),
\end{equation}
}where $D_{ii}, i = 1 \ldots N$ are square invertible matrices of size
$b_i \times b_i$ with $b_i < n$.  Let $D = \hbox{Block-Diag} (D_{11},
\ldots,D_{NN})$, and let 
\begin{displaymath}
L= \left(\begin{array}{cccc}
       0    &  &          &          \\
       L_{11} & \ddots    &    &          \\
       \vdots   & \ddots &   &  \\
       L_{N1} &  \ldots      & L_{N,N-1} & 0
       \end{array}
     \right),
\ \ 
    U = \left(\begin{array}{cccc}
       0    &  U_{11} &  \ldots       &  U_{1, N}   \\
       & \ddots    & \ddots  & \vdots     \\
              & & \ddots   &  U_{N-1, N}\\
              &        &  & 0
       \end{array}
     \right).
\end{displaymath}
The factorization (\ref{eq:LDU}) can be written as $A = ( L + D)
D^{-1} ( U + D)$.  In the following we refer to the blocks of $L, D,
U$ as $C$, where $L_{ij} = C_{ij}$ if $i > j$, $U_{ij} = C_{ij}$ if $i
< j$, and $D_{ij} = C_{ij}$ if $i = j$.  In other words, the matrix
$C$ can be written as $C = L + D + U$.  The blocks of $L, D$, and $U$
are computed using the following formula, with $i,j = 1 \ldots N$:
\begin{eqnarray}
\label{eq:blockLDU}
C_{ij} &=& 
\left\{
\begin{array}{ll}
A_{ij} \hspace{ .2cm}i= 1 \ or\ j=1 \\
A_{ij} - \sum_{k=1, L_{ik} \neq 0,U_{kj} \neq 0}^{\min(i,j)-1} L_{ik} D^{-1}_{kk}  U_{kj}, \hspace{.2cm} i > 1 \ or\ j > 1
\end{array}
\right.
\end{eqnarray}

In practice, even if the matrix $A$ is very sparse, the factors $L, D,
U$ can be much denser.  In particular the term $L_{ik} D^{-1}_{kk}
U_{kj}$ can introduce an important amount of fill-in in the factors.
In our work, the goal is to approximate the inverse of the diagonal
blocks $D^{-1}_{kk}, k = 1 \ldots n$ by a sparse matrix such that
$L_{ik} D^{-1}_{kk} U_{kj}$ stays sparse.  In the context of filtering
decomposition, there are mainly two approximations used.  Consider a
diagonal block $D_{kk}$.  The first approach consists of approximating
$D^{-1}_{kk}$ by a sparse matrix $\precf = \beta$, chosen such that a
filtering condition is satisfied.  The second approach aims at
identifying a better approximation of $D^{-1}_{kk}$ starting from
$\beta$.  As described in section \ref{sec:constrbeta}, this leads to
an approximation of the form $\precf = 2 \beta - \beta D_{kk} \beta$
\cite{Achdou:2007:LFT}, where $\beta$ is a diagonal matrix.  We will
discuss both approaches, but we note that the first approach is more
stable and leads to better results in practice.

In the following we explain the construction of the block filtering
preconditioner $M$.  We first give its definition, and then explain
more in detail the reasoning that lead to its construction.  In
section \ref{sec:constrbeta} we discuss the construction of the
approximation ${\precf}$ of the inverse of the block diagonal
matrices.

\begin{definition}
Let $A$ be a matrix of size $n \times m$.  For $k = 1 \ldots N$, let
$L_k$ be a matrix of size $n \times n_k$, $D_k$ be an invertible
matrix of size $n_k \times n_k$, and $U_k$ be a matrix of size $n_k
\times m$.  Let $M$ be a matrix defined by
\[
 M - A =  \sum_{k = 1}^{N} L_k D_k^{-1} U_k - \sum_{k = 1}^{N} L_k F_k U_k.
\]
A construction that enabes filtering is a construction where $F_k$, $k
= 1 \ldots N$ are matrices that satisfy the relation
\begin{equation}
F_{k} U_{k} t = D_{k}^{-1} U_{k} t  \ \ for\ all\ k=1:N 
\end{equation}
\end{definition}

\begin{definition}
\label{def:bfd}
Let $t$ be a filtering vector of size $n$ and let $A$ be a matrix of
size $n \times n$ partitioned into a block matrix of size $N \times
N$.  A block filtering decomposition is defined as
{\small
\begin{equation}
     \label{eq:BFD1}
M =\left(\begin{array}{cccc}
       {\precd}_{11}    &  &          &          \\
       {\precl}_{21} & {\precd}_{22}    &   &          \\
         \vdots     & \ddots & \ddots   &  \\
       {\precl}_{N1}       & \ldots       & {\precl}_{N,N-1} & {\precd}_{NN}
       \end{array}
     \right)
\left(\begin{array}{cccc}
       {\precd}_{11}^{-1}    &  &          &          \\
        & {\precd}_{22}^{-1}    &   &          \\
              & & \ddots   &  \\
              &        &  & {\precd}_{NN}^{-1}
       \end{array}
     \right)
\left(\begin{array}{cccc}
       {\precd}_{11}     & {\precu}_{12} & \ldots  & {\precu}_{1N} \\
        & \ddots  & \ddots  & \vdots \\
        &  & {\precd}_{N-1,N-1}    & {\precu}_{N-1,N} \\
              &        &  & {\precd}_{NN}
       \end{array}
     \right),
\end{equation}
}
where ${\precd}_{ii}, i = 1 \ldots N$ are square invertible matrices
of size $b_i \times b_i$ with $b_i < n$.  In more compact form, $M =
{\precl} {\precd} {\precu}$, where $M, {\precl}, {\precd}, {\precu}$
are block matrices of size $N \times N$.  Let ${\precc} = {\precl} +
{\precd} + {\precu}$ and let $t = (t_1; t_2; \ldots t_N)$.  The blocks
are computed as
\begin{eqnarray}
\label{eq:blockM}
{\precc}_{ij} &=& 
\left\{
\begin{array}{ll}
A_{ij} \hspace{ .2cm}i= 1 \ or\ j=1 \\
A_{ij} - \sum_{k=1, {\precl}_{ik} \neq 0,{\precu}_{kj} \neq 0}^{\min(i,j)-1} {\precl}_{ik} \precf_{kj} {\precu}_{kj}, \hspace{.2cm} i > 1 \ or\ j > 1
\end{array}
\right.
\end{eqnarray}
where $\precf_{kj}$ is a sparse approximation of $\precd_{kk}^{-1}$ that satisfies
\begin{equation}
\label{eq:formbeta}
\precf_{kj} {\precu}_{kj} t_j = {\precd}_{kk}^{-1} {\precu}_{kj} t_j  \ \ for\ all\ k=1:\min(i,j)-1 \ with\ {\precl}_{ik} \neq 0, {\precu}_{kj} \neq 0
\end{equation}
\end{definition}

If ${\precu}_{kj} t_j$ is a vector of nonzero elements, a matrix
$\precf_{kj}$ that satisfies the condition in equation
\eqref{eq:formbeta} can be computed as
\[
\precf_{kj} = Diag (({\precd}_{kk}^{-1} {\precu}_{kj} t_j) ./ {\precu}_{kj} t_j )
\]
where $./$ is the pointwise division (see for example equation (15) in
\cite{Achdou:2007:LFT}).  However, for very sparse matrices,
${\precu}_{kj}$ can have rows of all zeros, and hence the result of
${\precu}_{kj} t_j$ can be a vector with zero elements.  We present in
section \ref{sec:constrbeta} a construction of $\precf_{kj}$ that
solves this problem.

The main idea in the design of the preconditioner in Definition
\ref{def:bfd} is to ensure that each block satisfies an appropriate
filtering condition $M_{ij} t_j = A_{ij} t_j$, such that the global
filtering condition $Mt = At$ is satisfied, where $t = (t_1; t_2;
\ldots t_N)$ is the filtering vector.  We note $B = M-A$, and so we
want to ensure that for each block $B_{ij} t_j = 0$.  This is
different from the approach used for block tridiagonal systems
\cite{Achdou:2007:LFT}, where $B = M-A$ is a block diagonal matrix.
The matrix $B = M-A$ is formed by $(B_{ij})_{1 \leq i,j \leq N}$, with
\begin{equation}
\label{eq:blokcsB}
B_{ij} = {\precc}_{ij} + \sum_{k=1, {\precl}_{ik} \neq 0,{\precu}_{kj} \neq 0}^{\min(i,j)-1} {\precl}_{ik} {\precd}_{kk}^{-1} {\precu}_{kj} - A_{ij}.
\end{equation}
The construction of $M$ ensures that for each block $B_{ij}$, for each
term ${\precl}_{ik} {\precd}_{kk}^{-1} {\precu}_{kj}$ of the summation
in Equation \ref{eq:blokcsB}, $\precf_{kj}$ is chosen such that the
filtering is satisfied.  That is, ${\precl}_{ik} {\precd}_{kk}^{-1}
{\precu}_{kj} t_j = {\precl}_{ik} \precf_{kj} {\precu}_{kj} t_j$.
From this the formula of $\precf_{kj}$ in Equation \ref{eq:formbeta}
is deduced. This ensures that the global filtering for the whole
matrix is satisfied.  Note that there is a $\precf_{kj}$ for each
nonzero block ${\precu}_{kj}$, that is the approximation of the
diagonal block depends on the off-diagonal blocks of ${\precu}$.  We
give a formal proof in the following lemma.

\begin{lemma}
\label{lem:DefBFD}
Consider an $n \times n$ matrix $A$ and a filtering vector $t$ of size
$n$. If the block filtering preconditioner $M$ as defined in
Definition \ref{def:bfd} exists, then it satisfies the filtering
property, that is $Mt = At$.
\end{lemma}
\begin{proof}
The preconditioner $M$ satisfies the right filtering property if for
each nonzero block $B_{ij}$ we have $B_{ij} t_j = 0$, where $B_{ij}$
is of size $b_i \times b_j$ and $t_j$ is a vector of $b_j$
elements. In the formula of $B_{ij}$ from equation \eqref{eq:blokcsB},
we replace the expression of ${\precc}_{ij}$ from equation
\eqref{eq:blockM}.  We obtain:
\begin{eqnarray*}
B_{ij} t_j &=& \left( \sum_{k=1, {\precl}_{ik} \neq 0,{\precu}_{ki}
  \neq 0}^{\min(i,j)-1} {\precl}_{ik} {\precd}_{kk}^{-1} {\precu}_{kj}
- \sum_{k=1, {\precl}_{ik} \neq 0,{\precu}_{kj} \neq 0}^{\min(i,j)-1}
{\precl}_{ik} \precf_{kj} {\precu}_{kj} \right) t_j = \\
 &=& \left( \sum_{k=1, {\precl}_{ik} \neq 0,{\precu}_{ki} \neq
  0}^{\min(i,j)-1} {\precl}_{ik} 
{\precd}_{kk}^{-1} (I -{\precd}_{kk} \precf_{kj}  ) {\precu}_{kj}
\right) t_j = 0
\end{eqnarray*}
\end{proof}

We give now a definition of the block filtering preconditioner, in
which the inverse of a diagonal block matrix ${\precd}_{kk}$ is
approximated by $2 \precf_{kj} - \precf_{kj} {\precd}_{kk} \precf_{kj}$.
We show that if the matrix $\precf_{kj}$ satisfies the same condition
as in equation \eqref{eq:formbeta}, the preconditioner satisfies the
filtering property.
\begin{definition}
\label{def:bfd2}
A block filtering preconditioner $M$ of a matrix $A$ of size $n \times
n$ is defined by an incomplete block factorization 
{\small
\begin{equation}
     \label{eq:BFD2}
M =\left(\begin{array}{cccc}
       {\precd}_{11}    &  &          &          \\
       {\precl}_{21} & {\precd}_{22}    &   &          \\
         \vdots     & \ddots & \ddots   &  \\
       {\precl}_{N1}       & \ldots       & {\precl}_{N,N-1} & {\precd}_{NN}
       \end{array}
     \right)
\left(\begin{array}{cccc}
       {\precd}_{11}^{-1}    &  &          &          \\
        & {\precd}_{22}^{-1}    &   &          \\
              & & \ddots   &  \\
              &        &  & {\precd}_{NN}^{-1}
       \end{array}
     \right)
\left(\begin{array}{cccc}
       {\precd}_{11}     & {\precu}_{12} & \ldots  & {\precu}_{1N} \\
        & \ddots  & \ddots  & \vdots \\
        &  & {\precd}_{N-1,N-1}    & {\precu}_{N-1,N} \\
              &        &  & {\precd}_{NN}
       \end{array}
     \right).
\end{equation}
}
and a filtering vector $t$ of size $n$, where ${\precd}_{ii}, i = 1
\ldots N$ are square invertible matrices of size $b_i \times b_i$ with
$b_i < n$.  In more compact form, $M = {\precl} {\precd} {\precu}$,
where $M, {\precl}, {\precd}, {\precu}$ are block matrices of size $N
\times N$.  Let ${\precc} = {\precl} + {\precd} + {\precu}$ and let $t
= (t_1; t_2; \ldots t_N)$.  The blocks are computed as
\begin{eqnarray}
\label{eq:blockM2}
{\precc}_{ij} &=& 
\left\{
\begin{array}{ll}
A_{ij} \hspace{ .2cm}i= 1 \ or\ j=1 \\
A_{ij} - \sum_{k=1, {\precl}_{ik} \neq 0,{\precu}_{kj} \neq 0}^{\min(i,j)-1} {\precl}_{ik} (2 \precf_{kj} - \precf_{kj} {\precd}_{kk} \precf_{kj} ) {\precu}_{kj}, \hspace{.2cm} i > 1 \ or\ j > 1
\end{array}
\right.
\end{eqnarray}
where $\precf_{kj}$ is a sparse approximation of $\precd_{kk}$ that satisfies
\begin{equation}
%\beta_{kj} = Diag (({\precd}_{kk}^{-1} {\precu}_{kj} t_j) ./ {\precu}_{kj} t_j )
\precf_{kj} {\precu}_{kj} t_j = {\precd}_{kk}^{-1} {\precu}_{kj} t_j
\end{equation}
\end{definition}

\begin{lemma}
\label{lem:DefBFD2}
Consider an $n \times n$ matrix $A$ and a filtering vector $t$ of size
$n$. If the block filtering preconditioner $M$ as defined in
Definition \ref{def:bfd2} exists, then it satisfies the filtering
property, that is $Mt = At$.
\end{lemma}
\begin{proof}
We use the same approach as in Lemma \ref{lem:DefBFD}.  The
preconditioner $M$ satisfies the right filtering property if for each
nonzero block $B_{ij}$ we have $B_{ij} t_j = 0$. In the formula of
$B_{ij}$ from equation \eqref{eq:blokcsB}, we replace the expression
of ${\precc}_{ij}$ from equation \eqref{eq:blockM2}.  We obtain:
\begin{eqnarray*}
B_{ij} t_j &=& \left( \sum_{k=1, {\precl}_{ik} \neq 0,{\precu}_{ki}
  \neq 0}^{\min(i,j)-1} {\precl}_{ik} {\precd}_{kk}^{-1} {\precu}_{kj}
- \sum_{k=1, {\precl}_{ik} \neq 0,{\precu}_{kj} \neq 0}^{\min(i,j)-1}
{\precl}_{ik} (2 \precf_{kj} - \precf_{kj} {\precd}_{kk} \precf_{kj} )
{\precu}_{kj} \right) t_j = \\
 &=& \left( \sum_{k=1, {\precl}_{ik} \neq 0,{\precu}_{ki} \neq
  0}^{\min(i,j)-1} {\precl}_{ik} (\precf_{kj} {\precd}_{kk} - I)
{\precd}_{kk}^{-1} ({\precd}_{kk} \precf_{kj}  - I ) {\precu}_{kj}
\right) t_j = 0
\end{eqnarray*}
\end{proof}

\section{Construction of the approximation}
\label{sec:constrbeta}

We describe the construction of the approximation matrices
$\precf_{kj}$.  We denote the element in position $(i,j)$ of a matrix
$A$ as $A(i,j)$ and the element in position $i$ of a vector $v$ as
$v(i)$.

The block filtering preconditioner defined in Definitions
\ref{def:bfd} and \ref{def:bfd2} requires the construction of matrices
$\precf_{kj}$ that satisfy the equation \eqref{eq:formbeta}, that is
$\precf_{kj} {\precu}_{kj} t_j = {\precd}_{kk}^{-1} {\precu}_{kj} t_j$
We note in the following $M_{kj} t_j = v_{kj}$ and ${\precd}_{kk}^{-1}
{\precu}_{kj} t_j = u_{kj}$ , where $v_{kj}, u_{kj}$ are vectors of
$b_k$ elements. Hence we have $\precf_{kj} v_{kj} = u_{kj}$.  In the
following, for the ease of understanding, we simplify the notation and
discuss the relation $\precf v = u$.  The approach used previously for
the construction of $\precf$ is to compute it as
\[
\precf = Diag(u ./ v)
\]
where $./$ is pointwise division.

For sparse matrices, the vector $v$ can have zero elements. Possibly
$u$ can have only zero elements, but this case is simple to solve.  We
discuss the case of $v$ having zero elements.  If $v$ has only zero
elements, then $u$ is also zero, and hence the relation $\precf v = u$
is satisfied.  We discuss hence the case when there is at least a
nonzero element in $v$. Let $j$ be the index of a nonzero element,
that is $v(j) \neq 0$.  If $v(i) = 0$, then we take $\precf (i,j) = u(j)
/ v(j)$.  In other words, a simple construction of the matrix $\precf$
is as follows:
\begin{eqnarray}
\label{eq:Beta1}
\precf(i,j) &=& 
\left\{
\begin{array}{ll}
u(i) / v(i) \hspace{ .2cm} if\ i=j\ and\ v(i) \neq 0  \\
u(i) / v(j)  \hspace{ .2cm} if\ v(i) = 0 \ and\ j =
\min_{k: v(k) \neq 0}{\vert k-i \vert } \\ 
0 \hspace{ .2cm} otherwise \\
\end{array}
\right.
\end{eqnarray}

An example of construction of $\precf$ is as follows:
\[
\begin{pmatrix}
0 & u(1)/v(2) & & & & \\
 & u(2)/v(2) & & & & \\
 & u(3)/v(2) & 0 & & & \\
 &   &  & 0 & u(4)/v(5) & \\
 &   &  &  &  u(5)/v(5) & \\
 &   &  &  &  u(6)/v(5) & 0 \\
\end{pmatrix}
\cdot 
\begin{pmatrix}
0 \\
v(2) \\
0 \\
0 \\
 v(5) \\
0 \\
\end{pmatrix}
= 
\begin{pmatrix}
u(1) \\
u(2) \\
u(3) \\
u(4) \\
u(5) \\
u(6) \\
\end{pmatrix}
\]

The matrix $\precf$ can be easily constructed to be symmetric by
letting $\precf (j,i) = \precf(i,j)$.  But $\precf$ might not be SPD.

%Next step: find a construction that ensures $\beta$ is SPD.  For this,
%when $v(i) = 0$, we can choose any value on the diagonal $\beta
%(i,i)$.  It is possble to choose this value such that the Cholesky
%factorization of $\beta$ exists.

We can also use deflation techniques \cite{Nicolaides:1987:DCJ,
  MR1766015, MR2506657, MR2413778} to construct $\precf_{kj}$ that
satisfies the equation \eqref{eq:formbeta}, that is $\precf_{kj}
{\precu}_{kj} t_j = {\precd}_{kk}^{-1} {\precu}_{kj} t_j$.  Equation
\eqref{eq:deflation} defines $\precf_{kj}$ for a symmetric matrix.

\begin{eqnarray}
\label{eq:deflation}
\precf_{kj} &=& P+Q \\
P &=& I - QA \\
Q &=& Z E^{-1} Z^T \\
E &=& (Z^T \precd_{kk} Z)^{-1} \\
Z &=& \precu_{kj} t_j 
\end{eqnarray}

\subsection{Suitability for Parallelism}

The block filtering preconditioner was defined in a general way.  With
a suitable ordering, a parallel preconditioner can be obtained.  In
our work, we focus on matrices partitioned using nested dissection.
This partitioning leads to algorithms that can be implemented in
parallel.  We describe here briefly this ordering.  Nested dissection
considers the undirected graph $G$ of a symmetric matrix $A$.  It
identifies a separator $S$ that partitions the graph into two
disconnected graphs $G_1$, $G_2$.  The input matrix is permuted such
that the vertices corresponding to the separator $S$ are ordered after
the vertices corresponding to the two disconnected graphs $G_1,
G_2$. Then the same partitioning can be applied on the two
disconnected graphs, with the recursion beind stopped when the number
of desired independent parts has been reached.

Consider a matrix $A$ of size $n \times n$ partitioned using nested
dissection into a block matrix of size $N \times N$.  The following
example displays the result obtained after applying two steps of
nested dissection that leads to a block matrix of size $7 \times 7$.
{\small
\[
P A P^T = 
  \begin{pmatrix}
    A_{11} &   & A_{13} & & & & A_{17} \\
      & A_{22} & A_{23} & & & & A_{27} \\
    A_{31} & A_{32} & A_{33} & & & & A_{37} \\
 & & & A_{44} & & A_{46} & A_{47} \\
 & & & & A_{55} & A_{56} & A_{57} \\
 & & & A_{64} & A_{65} & A_{66} & A_{67} \\
A_{71} & A_{72} & A_{73} & A_{74} & A_{75} & A_{76} & A_{77} \\   
  \end{pmatrix},
\]
}
The preconditioner $M$ is defined as
{\small
\begin{equation}
M = 
\begin{pmatrix}
{\precd}_{11} & & & & & & \\
 & {\precd}_{22} & & & & & \\
{\precl}_{31} & {\precl}_{32} & {\precd}_{33} & & & & \\
 & & &{\precd}_{44} & & & \\
 & & & & {\precd}_{55} & & \\
 & & & {\precl}_{64} & {\precl}_{65} & {\precd}_{66} & \\
{\precl}_{71} & {\precl}_{72} & {\precl}_{73} & {\precl}_{74} & {\precl}_{75} & {\precl}_{76} & {\precd}_{77} \\
\end{pmatrix}
({\precd}_{ii}^{-1})_{i=1:7}
\begin{pmatrix}
{\precd}_{11} & & {\precu}_{13} & & & & {\precu}_{17} \\
 & {\precd}_{22} & {\precu}_{23} & & & & {\precu}_{27} \\
 &  & {\precd}_{33} & & & &  {\precu}_{37} \\
 & & & {\precd}_{44} & & {\precu}_{46} &  {\precu}_{74} \\
 & & & & {\precd}_{55} & {\precu}_{56} & {\precu}_{57} \\
 & & &  &  & {\precd}_{66} & {\precu}_{67} \\
 & & & &  &  & {\precd}_{77} \\
\end{pmatrix}
\end{equation}
} 
where each block of the factors $\precl, \precd, \precu$ can be
computed following Definitions \ref{def:bfd} or \ref{def:bfd2}.  With
this partition, both the preconditioner and the iterative process can
be implemented in parallel.  

\section{Conclusions}
\label{sec:concl}

In this report we have briefly presented a block filtering
preconditioner $M$ that is build from an input matrix $A$ and a
filtering vector $t$ and satisfies the property $M t = At$.  With an
appropriate ordering as nested dissection, this preconditioner is
suitable for parallel implementations.  A future paper will focus on
numerical results on scalar of PDEs discretized on two-dimensional and
three-dimensional structured and unstructured grids showing that this
method is efficient in practice.

\bibliographystyle{plain}
\bibliography{paperTFFD2004}

\end{document}